\documentclass[11pt,reqno]{amsart}
\usepackage[body={7in,9in},centering]{geometry}
\usepackage{amsmath,amssymb,amsthm,mathrsfs}
\usepackage{color,xcolor}
\usepackage{setspace}
\usepackage{parskip}
\definecolor{cobalt}{RGB}{61,99,181}

\usepackage[colorlinks,linkcolor=blue,anchorcolor=blue,citecolor=blue]{hyperref}

\newtheorem{thm}{Theorem}[section]
\newtheorem{cor}[thm]{Corollary}
\newtheorem{lem}[thm]{Lemma}
\newtheorem{rem}[thm]{Remark}
\newtheorem{ex}[thm]{Example}

\numberwithin{equation}{section}
\newtheorem{defi}{Definition}[section]

\date{\today}

\makeatletter

\newcommand{\Rmnum}[1]{\expandafter\@slowromancap\romannumeral #1@}
\makeatother

\newcommand{\N}{\mathbb{N}}
\newcommand{\Z}{\mathbb{Z}}

\newcommand{\fai}{\varphi}
\newcommand{\wci}{w C_\varphi}
\newcommand{\wcn}{(w C_\varphi)^n}
\newcommand{\you}{\rightarrow}

\begin{document}

\title[hypercyclic operators]{unbounded chaotic weighted pseudo-shift}

\author[ruxi Liang]{ruxi Liang\textsuperscript{1,2}}
\address{\textsuperscript{1} National Astronomical Observatories, Chinese Academy of Sciences, Beijing, 100101, P. R. China}
\address{\textsuperscript{2} University of Chinese Academy of Sciences, 19A Yuquanlu, Beĳing 100049, P. R. China}

\author[Pengyu Qin]{Pengyu Qin\textsuperscript{3}}
\address{\textsuperscript{3} College of Mathematics and Statistics, Chongqing University, Chongqing, 401331, P. R. China}

\author[yonglu Shu]{yonglu Shu\textsuperscript{*}}
\address{\textsuperscript{*} College of Mathematics and Statistics, Chongqing University, Chongqing, 401331, P. R. China}
\email{shuyonglu@163.com}
\keywords{Hypercyclic operator, Chaoticity, Pseudo-shift}

\begin{abstract}
In this paper, based on the work of Vijay K. Srivastava and Harish Chandra,~we give a characterization of the unbounded hypercyclic weighted pseudo-shift operator~$wC_{\varphi}$ on $\ell^p$~or~$c_0$.~Moreover we use the hypercyclicity criterion of Bès, Chan, and Seubert to give a necessary and sufficient condition in order that $wC_{\varphi}$ be a chaotic operator.
\end{abstract} \maketitle

\section{Introduction}\label{S1}
Let $X$ be be a topological vector space over the field $\mathbb{K}=\mathbb{R}$~or~$\mathbb{C}$,~a continuous linear operator $T$ on space $X$ is called $hypercyclic$ if there exists a vector $x\in X$ whose orbit $orb\left( x,T\right)=\left\{T^nx~|~n=0,1,2,...... \right\}$ is dense in $X$. Such a vector $x$ is said to be a hypercyclic vector for $T$, and the set of all hypercyclic vectors for $T$ is denoted by $HC(T)$. Finally, $T$ is said to be $chaotic$ if it is hypercyclic and its set of periodic points is dense in $X$.We refer to \cite{LinearChaos} for a more detailed study of hypercyclicity and chaoticity.

The study of linear dynamics has a long history and it is closely connected to the problem of invariant subspaces in Hilbert spaces.
In \cite{birkhoff1929demonstration}, Birkhoff showed the hypercyclicity of the operator of translation on the space $H(\mathbb{C})$of entire functions, while MacLane’s result\cite{maclane1952sequences} was on the hypercyclicity of the differentiation operator on $H(\mathbb{C})$. And later the first example of a hypercyclic operator on a Banach space was given by Rolewicz \cite{rolewicz1969orbits} in 1969, he showed that for any $\lambda$ with $|\lambda|>1$, the shift $\lambda B\left( x_n \right) _n=\left( \lambda x_{n+1} \right) _n$ are hypercyclic on the sequence space $\ell^2$. Motivated by these examples, the hypercyclicity and chaoticity of linear operators, as one of the most-studied properties in linear dynamics, has become an active area of research\cite{kitai1984invariant,gethner1987universal}.

Among these examples, the most-studied class is certainly that of weighted shifts. Following the paper of Rolewicz\cite{rolewicz1969orbits}, In 1995,~Salas \cite{salas1995hypercyclic} characterized hypercyclic and weakly mixing unilateral and bilateral weighted shifts on $\ell^2$ and $\ell^2\left(\mathbb{Z}\right)$ respectively while he characterized the supercyclic bilateral weighted shifts on Fr\'{e}chet space in \cite{salas1999supercyclicity}. The mixing weighted backward shift on $\ell^2$ and $\ell^2\left(\mathbb{Z}\right)$ were characterized by Costakis and Sambarino\cite{costakis2004topologically}. Furthermore, Martínez and Peris\cite{martinez2002chaos} gave new examples of chaotic backward shift operators in the special case of Köthe sequence spaces.

In \cite{salas1995hypercyclic}, the sufficient and necessary condition for the hypercyclicity of the weighted shift operator have been shown:
\begin{thm}
    Let $T$ be a continuous weighted backward shift on $\ell^2\left(\mathbb{N}\right)$ defined by $T\left(\sum_{k\in \mathbb{N}}{x_ke_k} \right) =\sum_{k\in \mathbb{N}}{w_kx_{k+1}e_k}$. Then $T$ is  hypercyclic if and only if there is an increasing sequence $\left( n_k \right) _{k\in \mathbb{N}}$~of positive integers such that the weight sequence $\left\{ w_k \right\} _{k\in \mathbb{N}}$ satisfies
\begin{align*}
    \prod_{j=1}^{n_k}{\left| w_j \right|}\rightarrow \infty , \ \ as\ n_k\rightarrow \infty.
\end{align*}
\end{thm}

This result is generalized by K.-G. Grosse-Erdmann to Fr\'{e}chet space in \cite{2000Grosse-Erdmann} and the characterization of chaotic weighted backward shift is provided in the following terms:

\begin{thm}
    Let $T$ be a weighted backward shift on a Fr\'{e}chet space $X$ in which $\left( e_n \right) _{n\in \mathbb{N}}$ is an unconditional basis. Then the following assertions are equivalent:\\
     \indent $(1)$  $T$ is chaotic.\\
     \indent $(2)$  $T$ has a nontrivial periodic point.\\
     \indent $(3)$  the series
           \begin{align*}
               \sum_{n=1}^{\infty}{\left( \prod_{k=1}^n{w_k} \right) ^{-1}}e_n
           \end{align*}
     converges in $X$.
\end{thm}

 Moreover, K.-G. Grosse-Erdmann proposed the weighted pseudo-shift operator and discussed its hypercyclicity in \cite{2000Grosse-Erdmann}. Let $w\in \ell^{\infty}$ and $\varphi$ be a self-map on the set of natural numbers $\mathbb{N}=\left\{1,2,3...... \right\} $, the weighted pseudo-shift on $\ell^p$ is defined by 
\begin{align*}
    wC_\varphi \left( x_n \right) _{n\in \mathbb{N}}=\left( w_nx_{\varphi \left( n \right)} \right) _{n\in \mathbb{N}}
\end{align*}
the operator $wC_\varphi$ is bounded if and only if $\underset{l\in \mathbb{N}}{\sup}\sum_{k\in \varphi ^{-1}\left( l \right)}{\left| w_k \right|^p <\infty}$. In \cite{Hypercyclicity},  Vijay K. Srivastava and Harish Chandra give a characterization of the hypercyclic weighted pseudo-shift as following:
\begin{thm}\label{youjieweiyiwei}
    The operator $wC_\varphi: \ell ^p \rightarrow \ell ^p$ is hypercyclic if and only if the following assertions hold: 
    \indent $(1)$ $\varphi$ is injective. \\
    \indent $(2)$ $\varphi ^n$ has no fixed point for any $n \in \mathbb{N} $. \\
    \indent $(3)$ For each $n \in \N$, there exists an increasing sequence $\{ m_k \}_{k \in \mathbb{N}}$ of positive integers such that 
    \begin{align*}
    \lim_{k\rightarrow \infty}\left| w_ n w_ {\varphi \left( n \right)} \cdots w_ {\varphi ^{m_k-1}\left( n \right)}  \right|=\infty .
    \end{align*}
\end{thm}

We need to note that the above studies are all for bounded operators. In fact, the unbounded hypercyclic operators have also attracted the interest of scholars. In \cite{bes2001chaotic},  Bès, Chan, and Seubert have defined the notion of hypercyclicity and chaoticity for an unbounded operator. In \cite{delaubenfels2003chaos}, R. deLaubenfels, H. Emamirad, and K.-G. Grosse-Erdmann generalized the notion of hypercyclic and chaotic semigroups to families of unbounded operators.
\begin{defi}\label{chaoxunhuandingyi}
Let $T$ be an unbounded operator on a separable infinite dimensional Banach space $X$, the densely defined operator $T$ is called hypercyclic if there is a vector $f \in D\left( T^{\infty}  \right)=\bigcap_{n\ge 1}{ D\left( T^n \right)}$~such that its orbit $\left\{ f, Tf, T^2f, T^3f,...... \right\}$ is dense in $X$. Every such vector $f$ is called a hypercyclic vector for $T$. If there exists an integer $n\in \mathbb{N}$ and a vector $f\in D\left( T^n \right) $ such that $T^nf=f$, such vector is called periodic and the operator A is said to be chaotic if it possesses both hypercyclic vector and dense sets of periodic points.
\end{defi}

Motivated by  \cite{Hypercyclicity}, Our aim in this paper is to extend the above result to the unbounded weighted pseudo-shift, establishing a necessary and sufficient condition so that a such operator could be hypercyclic or chaotic on $\ell^p$ and $c_0$ respectively.

To achieve this end, we must use a Hypercyclicity Criterion for unbounded operators. Such a criterion has been already provided by Bès, Chan, and Seubert \cite{bes2001chaotic} in the following terms:
\begin{thm}\label{wujiepanding}
    Let $X$ be a separable infinite dimensional Banach space and let $T$ be a densely defined linear operator on $X$ for which $T^n$ is a closed operator for all positive integers $n$. If there exist a dense subset $Y$ of the domain of $T$ and a (possibly nonlinear and discontinuous) mapping $S: Y \rightarrow Y$ such that \\
    \indent $(1)$ $TSy = y $, for all $y \in Y$,\\
    \indent $(2)$ $lim_{n \rightarrow \infty} S^n y = 0$, for all $y \in Y$,\\
    \indent $(3)$ $lim_{n \rightarrow \infty} T^n y = 0$, for all $y \in Y$,\\
    then $T$ is hypercyclic.
\end{thm}

Theorem \ref{wujiepanding} also holds when the entire sequence of positive integers in the hypothesis is replaced by a subsequence of positive integers. To prove hypercyclicity, we also need the following theorem given in \cite{LinearChaos}.
\begin{thm}\label{lc2.53}
    Let $T$ be a hypercyclic operator. Then its adjoint $T^*$ has no eigenvalues. Equivalently, every operator $T - \lambda I$, $\lambda \in \mathbb{K}$, has dense range.
\end{thm}

In fact, this theorem still holds for unbounded densely defined operators. In \cite{emamirad2005chaotic}, H. Emamirad and G.S. Heshmati characterized unbounded chaotic backward
shift operators on $\ell^2$ in the following terms:
\begin{thm}
    Let $T$ be an unbounded backward shift operator on  $\ell^2$. Then the following assertions are equivalent:\\
    \indent $(1)$  $T$ is chaotic.\\
    \indent $(2)$  $T$ has a nontrivial periodic point.\\
    \indent $(3)$  the positive series
    \begin{align*}
               \sum_{n=1}^{\infty}{\prod_{j=0}^{n-1}{\frac{1}{\left| w_j \right|^2}<+\infty}}.
           \end{align*}
\end{thm}

Based on the above theorems, we extend theorem \ref{youjieweiyiwei} to unbounded weighted pseudo-shift in section 2. And in section 3, we give a characterization of the chaotic weighted pseudo-shift operator on $\ell^p$ as follows:
\begin{thm}\label{hundunchongyao0}
         Let $wC_{\varphi} $~be an unbounded weighted pseudo-shift operator on $\ell^p$, then $wC_{\varphi} $ is chaotic if and only if the following assertions hold:  \\
        \indent $(1)$ $\varphi $ is injective.\\
        \indent $(2)$  $\varphi ^n$ has no fixed point for any $n \in \mathbb{N} $.\\
        \indent $(3)$  $\forall k \in \N$, the following inequalities hold
        \begin{equation*}
        \sum_{n=1}^{+\infty}{\frac{1}{\left| w_kw_{\varphi \left( k \right)}...w_{\varphi ^{n-1}\left( k \right)} \right|^p}}<+\infty ,
        \end{equation*}
        and
        \begin{equation*}
        \sum_{n=1}^{+\infty}{\left| w_{\varphi ^{-1}\left( k \right)}w_{\varphi ^{-2}\left( k \right)}...w_{\varphi ^{-n}\left( k \right)} \right|}^p<+\infty.
        \end{equation*}
    \end{thm}

    In section 3, we also extend the above result to the space $c_0$.

\section{Hypercyclicity}\label{S2}
In \cite{Hypercyclicity},  the necessary and sufficient condition for the hypercyclicity of bounded weighted pseudo-shift operators on $\ell^p$ is given. In this section, we extend it to the unbounded case. If the operator $wC_\fai$ is unbounded, then $\underset{l\in \mathbb{N}}{\sup}\sum_{k\in \varphi ^{-1}\left( l \right)}{\left| w_k \right|^p <\infty}$. We will see in below that the hypercyclicity of $wC_\fai$ implies that $\varphi$ is injective,i.e, $\underset{n\in \mathbb{N}}{\sup}~card\left\{ \varphi ^{-1}\left( n \right) \right\} =1$. So for the unbounded operator $wC_\fai$, we only concern the case when $\sup_{n\in \mathbb{N}}\left| w\left( n \right) \right|=\infty$. Denote
\begin{align*}
        e_{n} = \left( \underset{n - 1 }{\underbrace{0,\cdots ,0}}, 1, 0,\cdots \right) ,  \ n\in \N.
\end{align*}
and
\begin{align*}
    c_{00}=\left\{ x=\left( x_k \right) _{k\in \mathbb{N}}\in \mathbb{K}^{\mathbb{N}}\left| \exists N\in \mathbb{N},\ s.t.\ \forall k\ge N,\ x_k=0 \right. \right\}
\end{align*}
the space $c_{00}$ is dense in $\ell^p$ or $c_0$, we will always use this notation below. Let $\left\{ w_k \right\} _{k\in \mathbb{N}}$ be an arbitrary weight sequence, we define the unbounded weighted pseudo-shift $wC_\fai$ on $\ell^p$ by
\begin{align}
      wC_\fai (\sum_{k=1}^{\infty}{x_ke_k}) = \sum_{k=1}^{\infty}{w_k  x_{\fai (k)} e_k } 
    \end{align}
 where  $\underset{l\in \mathbb{N}}{\sup}\sum_{k\in \varphi ^{-1}\left( l \right)}{\left| w_k \right|^p <\infty}$, with the domain
\begin{align}
        D(wC_{\fai} ) = \left\{ x \in l^p \left |  \sum_{k=1}^{\infty}{w_k  x_{\fai(k)} e_k } \in l^p \right. \right\}  \text{.}
 \end{align}
By the definition of $wC_\fai$, 
\begin{align}
      \wcn (\sum_{k=1}^{\infty}{x_ke_k}) = \sum_{k=1}^{\infty}{w_k w_{\fai(k)} \cdots w_{\fai^{n-1} (k)} x_{\fai^{n} (k)} e_k } 
        \label{dingyi1n}
    \end{align}
   where  $\underset{l\in \mathbb{N}}{\sup}\sum_{k\in \varphi ^{-1}\left( l \right)}{\left| w_k \right|^p <\infty}$, with the domain
    \begin{align}
        D(wC_{\fai} ^n) = \left\{ x \in l^p \left |  \sum_{k=1}^{\infty}{w_k w_{\fai(k)} \cdots w_{\fai^{n-1} (k)} x_{\fai^{n} (k)} e_k } \in l^p \right. \right\}  \text{.}
        \label{dingyi2n}
    \end{align}
for all $n \in \mathbb{N}$.

Before proving the necessary and sufficient condition for the hypercyclicity of unbounded operator $wC_\fai$ on $\ell^p$, we need to show some lemmas.
\begin{lem}\label{faiwujie}
        Suppose map $\varphi :\ \mathbb{N}\rightarrow \mathbb{N}$ is injective and $\varphi ^n$ has no fixed point for any $n \in \mathbb{N} $. Then\\
        \indent $(1)$  $\forall n \in \N $ ,  $\fai^n (k) \you \infty \  (k \you \infty)$.\\
        \indent $(2)$  $\forall k \in \N $ ,  $\fai^n (k) \you \infty \  (n \you \infty)$.
\end{lem}
\begin{proof}
(1). Firstly, we show that $\varphi ^n$ is injective for any $n \ge 1$. Assume $\varphi ^{k-1}$ is injective, if there exists $i,j \in \mathbb{N}$ such that $\fai ^k (i) = \fai ^k (j)$~i.e $\fai(\fai ^{k-1} (i)) = \fai(\fai ^{k-1} (j))$. Since $\varphi$ is injective, then $\fai ^{k-1} (i) = \fai ^{k-1} (j)$. So by the assumption, $i=j$. Thus by induction, it follows that $\forall n \in \N$, $\fai ^n$ is injective.

 Conversely, suppose there exists $n_0 \in \mathbb{N}$ such that $\varphi ^{n_0}\left( k \right) \nrightarrow \infty \ as\ k\rightarrow \infty $, that is there exists a bounded subsequence $\{ \fai^{n_0}(m_k)\}_{k \in \N}$ of $\left\{ \fai^{n_0}(k) \right\}_{k\in \mathbb{N}}$. Suppose there exists $M \in \mathbb{N}$ such that $\forall k \in \N , \ \fai^{n_0}(m_k) \le M $. Consider the following set
 $$
    \{ \fai^{n_0}(m_1)\text{, } \fai^{n_0}(m_2)\text{, } \cdots \text{, } \fai^{n_0}(m_M)\text{, } \fai^{n_0}(m_{M+1}) \}
 $$
Note that each $\fai^{n_0}(k)$ is a positive integer, by the pigeonhole principle, there must exist positive integer $i,j \le M+1 \left( i\ne j \right)  $ such that $\fai ^{n_0} (m_i) = \fai ^{n_0} (m_j)$, which is a contradiction to the fact that $\fai ^{n_0}$ is injective.

\noindent (2). Conversely, suppose there exists $k_0 \in \mathbb{N}$ such that $\varphi ^{n}\left( k_0 \right) \nrightarrow \infty \ as\ n\rightarrow \infty $, that is there exists a bounded subsequence $\{ \fai^{m_k}(k_0)\}_{k \in \N}$ of $\left\{ \fai^{n}(k_0) \right\}_{n\in \mathbb{N}}$. Suppose there exists $M\in \mathbb{N}$ such that $\forall k \in \N , \ \fai^{m_k}(k_0) \le M $. Consider the following set
$$
    \{ \fai^{m_1} (k_0)\text{, } \fai^{m_2} (k_0)\text{, } \cdots \text{, } \fai^{m_M} (k_0)\text{, } \fai^{m_{M+1}} (k_0) \}
$$
By pigeonhole principle, there must exists positive integer $i,j \le M+1 \left( i\ne j \right)  $ such that $\fai ^{m_i} (k_0) = \fai ^{m_j} (k_0)$. This contradicts the fact that $\varphi ^n$ has no fixed point for any $n \in \mathbb{N} $.
\end{proof}

\begin{lem}\label{dingyiyu}
         Suppose map $\varphi :\ \mathbb{N}\rightarrow \mathbb{N}$ is injective and $\varphi ^n$ has no fixed point for any $n \in \mathbb{N} $. Then the subspace
        $$
        D\left( \left[ wC_{\fai}\right]^{\infty}  \right) \ =\ \bigcap_{n\ge 1}{ D\left( \left[ wC_{\fai} \right] ^n \right)}
        $$
        is dense in $\ell^p$.
\end{lem}
\begin{proof}
$\forall x=\sum_{k=1}^{\infty}{x_k e_k} \in c_{00}$, there exists $ N_1>0$ such that $\forall k>N_1,\ x_k=0$. Since $\varphi$ is injective, by Lemma \ref{faiwujie}  (1), there exists $ N_2>0$ such that $ \forall k>N_2,\ \fai^n(k)>N_1$. Take $N=\max \left\{ N_1,N_2 \right\}$, then $\forall k > N$, $x_k = x_{\fai^n (k)} = 0$.

Hence, $\forall n\in \N$
    \begin{equation*}
        \wcn (x) = \sum_{k=1}^{\infty}{w_k w_{\fai(k)} \cdots w_{\fai^{n-1} (k)} x_{\fai^{n} (k)} e_k } .
    \end{equation*}
has only finite term coefficients that are not zero, which implies $\wcn (x)\in l^p$. So $c_{00}\subset D\left( \left[  w C_{\fai} \right] ^n \right) \subset l^p$, and in view of the denseness of $c_{00}$ in $\ell^p$, it follows that $D\left( \left[  w C_{\fai} \right] ^n \right)$ is dense in $\ell^p$. By the arbitrary of $n$, the conclusion is proven.
\end{proof}

This lemma legitimates the research of hypercyclicity of the unbounded operator $wC_\fai$ in the framework of Definition \ref{chaoxunhuandingyi}. And in order to use the hypercyclicity criterion \ref{wujiepanding}, we have to ensure that for each $n\in \mathbb{N}$, the operator $\wcn$ is closed.
\begin{lem}\label{closed}
    $\forall n\in \mathbb{N}$,  the operator $\wcn$ defined by $\left(\ref{dingyi1n}\right)$ and $\left(\ref{dingyi2n}\right)$ is closed.
\end{lem}
\begin{proof}
suppose that as $j \rightarrow \infty$,
\begin{subequations}
    \begin{align}
        D\left( \left[ wC_{\fai} \right] ^n \right) \owns x_j = \sum_{k=1}^{\infty}{c_{k,j} e_k} &\rightarrow \sum_{k=1}^{\infty}{c_k^* e_k} = x^*\in l^p .
        \label{1shi}\\
        \wcn (x_j) &\rightarrow  \sum_{k=1}^{\infty}{d_k^* e_k} = f^* \in l^p.
        \label{2shi}
    \end{align}
    \end{subequations}
by $ \left( \ref{1shi} \right) $, we infer that $\forall k \in \N$ , $c_{k,j} \you c_k^* $ as $j \rightarrow \infty $. So
\begin{align*}
        \sum_{k=1}^{\infty}{w_k w_{\fai(k)} \cdots w_{\fai^{n-1} (k)} c_{\fai^{n} (k), j}} \you \sum_{k=1}^{\infty}{w_k w_{\fai(k)} \cdots w_{\fai^{n-1} (k)} c_{\fai^{n} (k)}^* } \  .
    \end{align*}
by $ \left( \ref{2shi} \right) $, 
\begin{align*}
        \wcn (x_j) = \sum_{k=1}^{\infty}{w_k w_{\fai(k)} \cdots w_{\fai^{n-1} (k)} c_{\fai^{n} (k), j} e_k } \you f^* = \sum_{k=1}^{\infty}{d_k^* e_k} \  .
    \end{align*}
Hence, 
    $$
    {w_k w_{\fai(k)} \cdots w_{\fai^{n-1} (k)} c_{\fai^{n} (k), j}} \you d_k^*.
    $$
It follows that
$$
        {w_k w_{\fai(k)} \cdots w_{\fai^{n-1} (k)} c_{\fai^{n} (k)}^*} = d_k^*.
$$
Thus,
\begin{align*}
        \wcn x^* &= \wcn (\sum_{k=1}^{\infty}{c_k^* e_k})\\
        &= \sum_{k=1}^{\infty}{w_k w_{\fai(k)} \cdots w_{\fai^{n-1} (k)} c_{\fai^{n} (k)}^* e_k }\\
        &= \sum_{k=1}^{\infty}{d_k^* e_k} = f^*.
\end{align*}
This implies that $x^*\in D\left( \left[ wC_{\fai} \right] ^n \right)$ and $\wcn x^* = f^*$. So we can conclude that the operator $\wcn$ is closed for each $n\in \mathbb{N}$. 
\end{proof}

\begin{thm}\label{danshe}
        Suppose $\forall n\in \mathbb{N}$, $\ w(n) \ne 0 $. If the unbounded weighted pseudo-shift operator $wC_\fai$ on $\ell^p$ is hypercyclic, then $\varphi$ is injective.
\end{thm}
\begin{proof}
Conversely, suppose there exists $k_1, k_2\in \mathbb{N} \left( k_1 \ne k_2\right)$  such that $ \fai(k_1) = \fai (k_2) = l$. Let the operator $\wci$ is hypercyclic, then the range $\text{Im}( wC_{\fai} ) $ is dense in $\ell^p$. So for $e_{k_1}\in l^p$, $\forall m \in \mathbb{N}$ ,  there exists ~a vector $  x_m \in D\left( \left[ wC_{\fai}\right]^{\infty}  \right)$ such that $\ \lVert \wci (x_m) - e_{k_1}\rVert \le \frac{1}{m}$ , that is 
\begin{align*}
        \lVert ( w_1x_{\fai (1)}, \cdots , w_{k_1-1}x_{\fai (k_1-1)}, w_{k_1}x_{\fai (k_1)} - 1, w_{k_1+1}x_{\fai (k_1+1)}, \cdots) \rVert \le \frac{1}{m}.
    \end{align*}
Hence,
\begin{subequations}
        \begin{align}
            |w(k_1)x_{\fai (k_1)} - 1| &\le \frac{1}{m} ,
            \label{shi1}\\
            |w(k_2)x_{\fai (k_2)}| &\le \frac{1}{m} .
            \label{shi2}
        \end{align}
     \end{subequations}
By $\left(\ref{shi1}\right)$,
\begin{align*}
        |x_l|=|x_{\fai (k_1)}|=\ge \frac{1}{|w_{k_1}|} (1-\frac{1}{m}) \you \frac{1}{|w_{k_1}|} \ ( m \you \infty).
\end{align*}
By $\left(\ref{shi2}\right)$,
\begin{align*}
        |x_l|=|x_{\fai (k_2)}|\le \frac{1}{m|w_{k_2}|} \you 0 \ ( m \you \infty).
 \end{align*}
This is a Contradiction, which completes the proof.
\end{proof}

\begin{thm}\label{nofixedpoint}
        If the unbounded weighted pseudo-shift operator $wC_\fai$ on $\ell^p$ is hypercyclic, then $\varphi ^n$ has no fixed point for each $n \in \mathbb{N} $.
\end{thm}
\begin{proof}
Let $wC_\fai$ be hypercyclic.  Conversely, supposes $\varphi ^n$ has some fixed points for each $n \in \mathbb{N} $, then there exists $ l,n\in \mathbb{N}$ such that $\fai^n (l) = l $. Let $f_l(x) = x_l $,  then $f_l$ is a bounded linear functional on $\ell^p$. By the definition of the adjoint $\left( wC_{\varphi} \right) ^*$,
 \begin{align*}
        \left( \left( \left( wC_{\fai} \right) ^n \right) ^{\ast}\left( f_l \right) \right) (x)
        &= \left( f_l \left( \left( wC_{\fai} \right) ^n \right) (x) \right)\\
        &=  f_l \left( \left( wC_{\fai} \right) ^n (x) \right) \\
        &=  f_l \left( \left( w_k w_{\fai(k)} \cdots w_{\fai^{n-1} (k)} x_{\fai^{n} (k)} \right)_{k \in \N} \right)\\
        &=  w_l w_{\fai(l)} \cdots w_{\fai^{n-1} (l)} x_{\fai^{n} (l)}\\
        &=  w_l w_{\fai(l)} \cdots w_{\fai^{n-1} (l)} x_l\\
        &=  w_l w_{\fai(l)} \cdots w_{\fai^{n-1} (l)} f_l(x) .
    \end{align*}
    Hence,
\begin{equation}
        \left( \left( \left( wC_{\fai} \right) ^n \right) ^{\ast}\left( f_l \right) \right) =
        w_l w_{\fai(l)} \cdots w_{\fai^{n-1} (l)} f_l(x) .
\end{equation}
 It follows that $f_l$ is an eigenvector of $\wcn$. Since $wC_{\varphi} $ is hypercyclic, then $\left( wC_{\varphi} \right) ^n$ is also hypercyclic, which is a contradiction to Theorem \ref{lc2.53}.
 \end{proof}

Before proving our conclusion, we need to extend the definition of $\fai$. $\forall n\in \N$, we define
\begin{align} \label{notation}   
 \varphi ^{-n}\left( k \right) =\begin{cases}
	l ,& \exists \, l \,\, \text{s.t.} \, \varphi^n\left(l\right) = k,\\
	0 ,& \varphi ^{-n}\left( k \right) =\varnothing.
    \end{cases}
\end{align}
Moreover, we specify $e_0=(0,0,\cdots,0,\cdots)$ and $w_0=0$. By the definition of $wC_\fai$, 
\begin{align}
        \wcn (e_k) &=  w_{\fai^{-n}(k)} w_{\fai^{-(n-1)}(k)} \cdots w_{\fai^{-1}(k)} e_{\fai^{-n}(k)}.
        \label{wnekdiedai}
\end{align}
for all $n\in \mathbb{N}$.

Next, our conclusions are given.

\begin{thm}\label{wujiew}
        The unbounded operator $wC_\fai $ on $\ell^p$ is hypercyclic if and only if the following assertions hold: \\
        \indent $(1)$ $\fai$ is injective. \\
        \indent $(2)$ $\fai ^n$ has no fixed point for any $n \in \mathbb{N} $. \\
        \indent $(3)$ For each $n \in \N$, there exists an increasing sequence $\{ m_k \}_{k \in \mathbb{N}}$ of positive integers such that 
        \begin{equation*}
        \lim_{k\rightarrow \infty}\left| w\left( n \right) w\left( \fai \left( n \right) \right) \cdots w\left( \fai ^{m_k-1}\left( n \right) \right) \right|=\infty .
        \end{equation*}
    \end{thm}
\begin{proof}
Assume $ w C_{\fai}$ is hypercyclic. By Theorem \ref{danshe} and Theorem \ref{nofixedpoint}, we infer that the condition $\left( 1 \right) \text{and} \left( 2 \right) $ hold. Now we need only to show $\left( 3 \right)$. Since the set of hypercyclic vectors for $ w C_{\fai}$ is dense in $\ell^p$, then $\forall \delta _k>0$, there exists a vector $\ x\in HC\left(  w C_{\fai} \right)$ (also belongs to $D\left( \left[  w C_{\fai} \right] ^\infty \right)$ ) such that $\ \lVert x-e_n \rVert <\delta _k$. Therefore, $\forall k\ne n$,
\begin{equation*}
        \left| x_k \right|<\delta _k\Rightarrow \left| x_{\fai ^j\left( n \right)} \right|<\delta _k, ~\forall j\in N \  \left( \varphi ^n\ has\ no\ fixed\ point\ \right) 
    \end{equation*}
Since $x\in HC\left(  w C_{\fai} \right) $, then $orb\left( x, w C_{\fai} \right) \overset{dense}{\subset}l^p$. So there exists $m_k>0$ such that
 $$\lVert \left(  w C_{\fai} \right) ^{m_k}x-e_n \rVert <\delta_k , $$
 Hence,
 \begin{align*}
        &\phantom{< \ } \left|  w _n w _{\fai \left( n \right)} \cdots w _{\fai ^{m_k-1}\left( n \right)}x_{\fai ^{m_k}\left( n \right)}-1 \right| \\
        &< \left(  |w _n w _{\fai \left( n \right)} \cdots w _{\fai ^{m_k-1}\left( n \right)}x_{\fai ^{m_k}\left( n \right)}-1| ^p +  \sum_{k \in \N, k \ne n}{|w _n w _{\fai \left( n \right)} \cdots w _{\fai ^{m_k-1}\left( n \right)}x_{\fai ^{m_k}\left( n \right)}}|^p \right) ^{1/p}\\
        & = \lVert \left(  w C_{\fai} \right) ^{m_k}x-e_n \rVert \\
        &< \delta _k.
        \label{xlpfanshu1}
    \end{align*}
Therefore,
$$1-\left|  w _n w _{\fai \left( n \right)} \cdots w _{\fai ^{m_k-1}\left( n \right)}x_{\fai ^{m_k}\left( n \right)} \right|<\delta _k ,$$
Thus,
$$\left|  w _n w _{\fai \left( n \right)} \cdots w _{\fai ^{m-1}\left( n \right)} \right|>\frac{1-\delta _k}{\left| x_{\fai ^{m}\left( n \right)} \right|}>\frac{1-\delta _k}{\delta _k} = \frac{1}{\delta_k} - 1.$$
By the arbitrary of $\delta_k$, we can choose a sequence  $\left\{ \delta _k \right\} $ which tends to 0, so the condition  $\left( 3 \right)$ is necessary if $w C_{\fai}$ is hypercyclic.  

For the sufficient part of the proof, we will use the Hypercyclicity Criterion stated in Theorem \ref{wujiepanding}. The closedness of $\wcn$ is already proved in Lemma \ref{closed}. Assume the condition $\left( 1 \right) ,\left( 2 \right) ,\left( 3 \right)$ hold, now we define the map
  \begin{align*}
S:\text{\quad } c_{00} & \rightarrow c_{00}\\
\left( e_n \right) & \rightarrow  \frac{e_{\fai \left( n \right)}}{ w _n},
\end{align*}
As is easily seen, $\forall n\in N$
\begin{align*}
        \left( wC_{\fai} \right) \left( S\left( e_n \right) \right) &= w C_{\fai}\left( \frac{e_{\fai \left( n \right)}}{ w _n} \right) =e_n .
\end{align*}
$\forall x=\sum_{k=1}^{\infty}{x_k e_k} \in c_{00}$, there exists $ N_1>0$ such that $\forall k>N_1,\ x_k=0$. In view of the condition $\left( 1 \right) ,\left( 2 \right) $ and Theorem \ref{faiwujie} $\left( 2 \right)$, for each $k\in \mathbb{N}$, there exists $N_2>0$ such that $\forall n>N_2,\ \fai^n(k)>N_1$. So take $N=\max \left\{ N_1,N_2 \right\}$, we infer that for any $n>N$, $ x_{\fai ^n\left( k \right)}=0 $. So for each $ x\in c_{00}$,
 \begin{equation}
        \left( \left(  w C_{\fai} \right) S \right) \left( x \right) =\sum_{k=1}^{\infty}{x_k [ \left( wC_{\fai} \right) \left( S\left( e_k \right) \right)]}
        =\sum_{k=1}^{\infty}{x_k e_k}
        = x
        \label{zhunze1}
    \end{equation}
and
\begin{equation}
        \wcn (x) = \sum_{k=1}^{\infty}{w_k w_{\fai(k)} \cdots w_{\fai^{n-1} (k)} x_{\fai^{n} (k)} e_k } \rightarrow 0\ \ \left( n\rightarrow \infty \right) .
        \label{zhunze2}
    \end{equation}
Further, by the definition of $S$,
\begin{align*}
        \lVert S^nx \rVert &= \left\lVert S^n\left( \sum_{k=1}^N{x_ke_k} \right) \right\rVert \\
        &=\left \lVert \sum_{k=1}^N{x_k \frac{e_{\fai ^n\left( k \right)}}{ w _k w _{\fai \left( k \right)} \cdots w _{\fai ^{n-1}\left( k \right)}}} \right \rVert\\
        &=\left( \sum_{k=1}^N{\left| x_k \right|^p\frac{1}{\left| w_kw_{\varphi \left( k \right)}...w_{\varphi ^{n-1}\left( k \right)} \right|^p}} \right) ^{\frac{1}{p}}\\
        &\le \lVert x \rVert \left( \sum_{k=1}^N{\frac{1}{\left| w_kw_{\varphi \left( k \right)}...w_{\varphi ^{n_k-1}\left( k \right)} \right|^p}} \right) ^{\frac{1}{p}}.
\end{align*}
Now in view of the condition on $w$ that for each $n \in \N$, there exists an increasing sequence $\{ m_k \}_{k \in \mathbb{N}}$ of positive integers such that 
\begin{equation*}
       \lim_{k\rightarrow \infty}\left| w\left( n \right) w\left( \fai \left( n \right) \right) \cdots w\left( \fai ^{m_k-1}\left( n \right) \right) \right|=\infty .
\end{equation*}
for each $x\in c_{00}$, we obtain that 
\begin{equation}
        \ S^{m_k} x \rightarrow 0\ ,\left( k \rightarrow \infty \right).
        \label{zhunze3}
\end{equation}
Hence with (\ref{zhunze1}),(\ref{zhunze2}),(\ref{zhunze3}), $w C_{\fai}$ satisfies the Hypercyclicity Criterion stated in Theorem \ref{wujiepanding}, thus $w C_{\fai}$ is hypercyclic.
\end{proof}

As is easily seen that the characterization of the hypercyclic weighted pseudo-shift operator $w C_{\fai}$ is the same in both the bounded and unbounded cases. So using the same approach, we can extend the conclusion to weighted pseudo-shift operators (whether bounded or unbounded) on the space $c_0$.

\begin{cor}\label{c0chaoxunhuan}
        The weighted pseudo-shift operator $wC_\fai $ on $ c_0$ is hypercyclic if and only if the following assertions hold: \\
        \indent $(1)$ $\fai$ is injective. \\
        \indent $(2)$ $\fai ^n$ has no fixed point for any $n \in \mathbb{N} $. \\
        \indent $(3)$ For each $n \in \N$, there exists an increasing sequence $\{ m_k \}_{k \in \mathbb{N}}$ of positive integers such that 
        \begin{equation*}
        \lim_{k\rightarrow \infty}\left| w\left( n \right) w\left( \fai \left( n \right) \right) \cdots w\left( \fai ^{m_k-1}\left( n \right) \right) \right|=\infty .
        \end{equation*}
 \end{cor}

\section{Chaoticity}\label{S3}
  In this section, we will give a necessary and sufficient condition in order that $w C_{\fai}$ be a chaotic operator. Some preparation is needed before proving it. The discussion of chaoticity requires us to pay more attention to the properties of $\fai$, so we will prove some lemmas in the first section.

\subsection{Some Lemmas} \leavevmode



The following lemma is provided as a supplement to Lemma \ref{faiwujie}.

\begin{lem}\label{faiwujie2}
        Suppose $\fai$ is injective and $\fai ^n$ has no fixed point for any $n \in \mathbb{N} $, then for any $k\in \N$, if $\forall n \in \N, \ \fai^{-n}(k) \ne 0$,  then $\fai^{-n} (k) \you \infty \  (n \you \infty)$.
\end{lem}
\begin{proof}
Conversely, suppose there exists $k_0 \in \mathbb{N}$ such that $\fai^{-n} (k_0)\nrightarrow \infty \  (n \you \infty)$, then there exists a bounded subsequence $\{ \fai^{-m_k}(k_0)\}_{k \in \N}$ of $\left\{ \varphi ^{-n}\left( k_0 \right) \right\} _{n \in \N}$. Suppose there exists $M\in \mathbb{N}$ such that $\forall k \in \N , \ \fai^{-m_k}(k_0) \le M $. Consider the following set
$$
    \{ \fai^{-m_1} (k_0)\text{, } \fai^{-m_2} (k_0)\text{, } \cdots \text{, } \fai^{-m_M} (k_0)\text{, } \fai^{-m_{M+1}} (k_0) \}
$$
By pigeonhole principle, there must exists positive integer $i,j \le M+1 \left( i\ne j \right)  $ such that $\fai ^{-m_i} (k_0) = \fai ^{-m_j} (k_0)$. This contradicts the fact that $\varphi ^n$ has no fixed point for any $n \in \mathbb{N} $.
\end{proof}

For each $n \in \N$ and any positive integer $k$, we denote a ``$\fai^n$ bi-orbit" (later referred to as  ``$\fai^n$-orbit"  ) of $k$ by the set
\begin{align*}
        A_{n,k} = orb(k,\fai^n) \cup orb(k,\phi^{n})= \{\fai^{in}(k) | i \in \Z \} - \{ 0 \},
\end{align*}
Where $\phi=\fai^{-1}$. Moreover, we denote $A_k = A_{1,k}$ and $B_{n,k} = \{\fai^{in}(k) | i \in \Z , i\ne 0 \} - \{ 0 \} = A_{n,k} - \{ k \}$.

The reason for dropping $\left\{ 0 \right\}$ from the definition of $A_{n,k}$ is to satisfy the previous definition in $\left( \ref{notation} \right)$  which will make it easier for us to select the minimum element $A_{n,k}$. By definition, it is easy to know the set $A_{n,k} \subseteq \N$. 

When $\varphi$  is injective, for any fixed $k$, there are two cases: when $k $ is a fixed point of $\varphi$ for some $n$, $A_{k}$ is a finite set; when $k $ is not fixed point of $\varphi$ for any $n$, $A_{k}$ is an infinite set. It is worth noting that there are also two cases in which $A_{k}$ is an infinite set: one is when there exists $ n \in \N$ such that $\fai^{-n}(k) = 0$ i.e there is a point without any preimage under $\fai^{n}$ in the orbit $A_{k}$ (this point is like the ``generator" of the orbit); and the other is when $\fai^{-n}(k) \ne 0$ for any $ n \in \N$, that is, there is no point without any preimage under $\fai^{n}$ in the orbit $A_{k}$. The above statement also holds for $A_{n,k}$.

Next, we give some examples of $\varphi$ for better illustration.

\begin{ex}
        Consider the mapping $\varphi \left( n \right) =\begin{cases}
	2,&		n=1\\
	3,&		n=2\\
        1,&		n=3\\
        4,&     \text{else}
        \end{cases}$.
    \end{ex}
In this case, $\forall k=1,2,3 $, $A_k=\left\{1,2,3\right\} $ is a finite set and $k$ is a fixed point of $\varphi$.

\begin{ex}\label{li_njia1}
     Consider the mapping $\varphi \left( n \right)=n+1 $,  $\forall n \in \mathbb{N} $.
\end{ex}
For $k=1, \fai^{-1}(1) = 0$ i.e $1$ has no preimage. In this case, $A_1=\mathbb{N}$ is an infinite set and there is only one $\fai$-orbit  $A_1$, that is, for any $ k \in \N, A_k=A_1$. Furthermore, $\fai ^n$ has no fixed point for any $n \in \mathbb{N} $.

\begin{ex}\label{failizi}
     Consider the mapping $\varphi \left( n \right) =\begin{cases}
	5k-4,&		n=5k+1\\
	5k+2,&		n=5k ~~~~~~ \left(k\ge 1\right) \\
        n+1,&     \text{else}
        \end{cases}$.
\end{ex}
For any $k,n\in \mathbb{N}, \  \fai^{-n}(k) \ne  0$, that is, there is no point without any preimage. In this case, there is only one $\fai$-orbit  $A_1$. Moreover, $A_1=\mathbb{N}$ is an infinite set, and $\fai ^n$ has no fixed point for any $n \in \mathbb{N} $.

\begin{lem}\label{n1n2}
    Suppose $\fai$ is injective without any periodic point, then for each $k$ and $N$,\\
    \indent $(1)$  $\forall n \in orb(k,\fai^N)$, let $\varphi^{n_1N}\left( k \right) =n (n_1 \in \N)$, then $\lim_{n\you \infty}{n_1}=\infty$;\\
    \indent $(2)$  $\forall n \in orb(k,\phi^N)$, let $\varphi ^{-n_2N}\left( k \right) =n (n_2 \in \N)$, then $\lim_{n\you \infty}{n_2}=\infty$.
\end{lem}
\begin{proof}
    (1). For any positive integer $M$, let $L = \{\fai^{n_1N}|n_1\le M\}$ and $N = \sup L$, then when $n \in orb(k,\fai^N)$ and $n>N$, we have $n_1>M$. This is because if $n_1<M$, then $n = \fai^{n_1N}\in L$, hence $n\le \sup L= N$, which makes a contradiction. Therefore, we have $\lim_{n\you \infty}{n_1}=\infty$.
    
    (2). For any $M \in \N$, let $L = \{\fai^{-n_2N}|n_2\le M\}$ and $N = \sup L$, then when $n \in orb(k,\phi^N)$ and $n>N$, we have $n_2>M$. This is because if $n_2<M$, then $n = \fai^{-n_2N}\in L$, hence $n\le \sup L= N$, which makes a contradiction. Therefore, we have $\lim_{n\you \infty}{n_2}=\infty$.
\end{proof}

The following lemma shows that the different $\fai^n$ orbits are disjoint when $\varphi$ is injective.

 \begin{lem}\label{guidaobujiao}
        Suppose $\fai$ is injective, then for all  distinct $ k_1, k_2 \in \N$ , their $\fai^n$-orbit either the same or disjoint. That is, if $A_{n,k_1} \ne A_{n,k_2}$, then $A_{n,k_1} \cap A_{n,k_2} = \varnothing.$
\end{lem}
\begin{proof}
For any fixed $n$, suppose there exists $k_1, k_2 \in \N$ such that $A_{n,k_1} \cap A_{n,k_2} \ne \varnothing$,that is, there exists $ i,j \in \Z$ such that $\fai^{in}(k_1) = \fai^{jn}(k_2) = l$. If $i=j$, then by $\varphi$ being a injection we infer that $k_1 = k_2$ , therefore $A_{n,k_1} = A_{n,k_2}$. If $i \ne j$, we might assume that $i>j$, then $k_2 = \fai^{(i-j)n}(k_1)$, $k_2$ is in the orbit of $k_1$ ,i.e , $A_{n,k_1} = A_{n,k_2}$.
\end{proof}

Define $g_{n}(k)=\inf \{ A_{n,k} \}, \ k \in \N$, this is a mapping from $\N$ to $\N$  which maps each $k \in \N$ to the smallest element in $A_{n,k}$ . We may regard the orbit $A_{n,k}$ to be generated by $g_{n}(k)$ under $\fai^{n}$. Moreover, we let $g(k)=g_{1}(k)=\inf \{ A_k\}$.

Denote $G_n = \text{Im} (g_n)$ as the range of $g_{n}$, which represents the set of ``generator" for each $\fai^n$ orbit, which may be finite or infinite, but must be countable. Moreover,  we let $G=G_1=\text{Im}(g)$. The main purpose for this definition is to select the ``generator" of each orbit.

Again, we give some examples for illustration.

\begin{ex}
        Consider the mapping defined in $\left(\ref{li_njia1}\right)$ and $\left(\ref{failizi}\right)$.
\end{ex}
For these two mappings, $A_k=A_1=\mathbb{N},$ and $ \  g(k)=1$ $\left( \forall k\in \mathbb{N}\right)$. Thus $G=\{ 1 \} $ with $\left| G \right|=1$.

\begin{ex}
         Consider the mapping $\varphi \left( n \right) =\begin{cases}
	5,&		   n=1\\
        5k+3,&     n=5k-1\\
	5k+5,&	   n=5k ~~~~~~ \left(k\ge 1\right) \\
        5k-4,&     n=5k+1\\
        5k-3,&     n=5k+2\\
        n+1,&     \text{else}
        \end{cases}$.
\end{ex}
  In this case, $\mathbb{N}$ is divided into two $\fai$-orbit, $A_1=\left\{ 1,5k,5k+1 \right\} \left( k\ge 1 \right) $ and $A_2=\mathbb{N}-A_1$. Thus $G=\{ 1,2 \}$ with $ \left| G \right|=2$.

\begin{ex} \label{sushulizi}
        Let $P_i=\left\{ p_i^n \left| ~p_i\ is\ the\ i\text{-}th\ prime\ number, \ \forall i,n \in \N \right. \right\}$ and  $
   P_0=\mathbb{N}-\bigcup_{i\ge 1}{P_i}$. Then we arrange the elements in $P_i$ in ascending order and denote $P_i\left( m \right)$ as the $m$-th element. It is clear that $P_i\left( m \right) = p_i^m (i\ge 1)$.
        
        1.Consider the mapping $\varphi \left( P_i\left( k \right) \right) =P_i\left( k+1 \right) ,\ \forall i\ge 0$. 
        
        In this case, Since there are infinitely many prime numbers, then $\mathbb{N}$ is divided into a countably infinite number of orbits. And there are points with no preimage in each orbit, actually for any $i$, $\fai^{-1}(P_i(1))=0$. Furthermore,  for any $k \in \N$, we can infer that $g(P_i(k))=p_i  \ (i\in \N)$ and $g(P_0(k))=1$. In this case, $G$ is the set of all prime numbers and 1, Thus $\left| G \right|=+\infty $.

        2. Consider the mapping  $\varphi \left( n \right) =\begin{cases}
	P_i\left( 2 \right) ,&		n=P_i\left( 1 \right)\\
	P_i\left( 3k \right) ,&		n=P_i\left( 3k-1 \right)\\
        P_i\left( 3k+2 \right) ,&		n=P_i\left( 3k \right)\\
        P_i\left( 3k-2 \right) ,&		n=P_i\left( 3k+1 \right)\\
        \end{cases}$.
        
        In this case, $\mathbb{N}$ is divided into a countably infinite number of orbits. And there are points with no preimage in each orbit. Moreover, for any $k \in \N$, we infer that $g(P_i(k))=p_i \ (i\in \N)$ and $g(P_0(k))=1$. Thus, $G$ is the set of all prime numbers and 1 with $\left| G \right|=+\infty $.

\end{ex}
The following lemma shows that we can divide the positive integers into some disjoint $\fai^n$ orbits.

 \begin{lem}\label{guidaobianli}
        Suppose $\fai$ is injective, then for $\forall n \in N$,we have
        \begin{align*}
            \N = \bigsqcup_{k\in G_n}{A_{n,k}}.
        \end{align*}
    \end{lem}
\begin{proof}
By Lemma \ref{guidaobujiao}, we know that each orbit $A_{n,k}$ is disjoint. Let $S = \bigsqcup_{k\in G_n}{A_{n,k}}$, it suffice to show that $S \subseteq \N$ and $\N \subseteq S$. Clearly, $S \subseteq \N$ since each $A_{n,k} \subseteq \N$. On the other hand, $\forall l \in \N$, there exists $k \in \N$ such that $g_n(l)=k\in G_n$, then $l \in A_{n,k} \subseteq S$, hence $\N \subseteq S$.
\end{proof}

\begin{lem}\label{Glian}
        Suppose $\fai$ is injective, and $\fai$ has no periodic point, then for every $k,n \in \N$, 
        \begin{align*}
            A_k &= A_{2,k} \sqcup A_{2,\fai(k)}\\
                &= A_{3,k} \sqcup A_{3,\fai(k)} \sqcup A_{3,\fai^2(k)}\\
                &= \bigsqcup\limits_{i=0}^{n-1}{A_{n,\varphi ^i\left( k \right)}}.
        \end{align*}
        So we have $ G \subset G_2 \subset G_3 \subset \cdots$.
\end{lem}
\begin{proof}
By Theorem \ref{guidaobujiao}, we know that each orbit $A_{n,k}$ is disjoint. Let $S=\bigsqcup\limits_{i=0}^{n-1}{A_{n,\varphi ^i\left( k \right)}}$. For any $k,n \in \N$, $\forall l \in A_k$, there exists $m\in \Z$ such that $l = \fai^{m}(k)$. 
Also there exists $a\in \Z$ and $p \in \{0,1,2,\cdots,n-1 \}$ such that $m = an + p$, it follows that $l = \fai^{m}(k) = \fai^{an + p}(k) = \fai^{an}(\fai^{p}(k)) \in A_{n,\fai^{p}(k)}$. Thus $A_k \subseteq S$. 

On the other hand, for any $l \in S$, there exists $a\in \Z, \ p \in \{0,1,2,\cdots,n-1 \}$ such that $l = \fai^{an}(\fai^{p}(k))$, let $m = an + p$, then $l = \fai^{an + p}(k) = \fai^{m}(k) \in A_k$, so $S \subseteq A_k$ which completes the proof. 
\end{proof}

\begin{lem}\label{taoyi}
        Suppose $\fai$ is injective, and $\fai$ has no periodic point, then for any finite set $F \subseteq  \N$, and every $k \in \N$, there exists a $N \in \N$ such that for any $ n > N$,
        \begin{align*}
            B_{n,k} \cap F = \varnothing. \qedhere
        \end{align*}
\end{lem}
\begin{proof}
For any finite set $F \subseteq  \N$, there exists $M>0$ such that  $ \max\left\{ F \right\} <M$.
 By Lemma \ref{faiwujie} (2), $\forall k\in \N$, there exists $N_1 \in \N$ such that for any $ n>N_1$, $\fai^n(k)>M$. Therefore for any  $ n>N_1$ and $i \in \N$,  $\fai^{in}(k)>M$ which means $\left(orb(k,\fai^n) - \{ k \}\right) \cap C = \varnothing$. By Lemma \ref{faiwujie2}, there exists $ N_2 \in \N$ such that for any $ n>N_2$, $\fai^{-n}(k)>M$, then for any $n>N_2$ and $ i \in \N$, $\fai^{-in}(k)>M$ which means $\left(orb(k,\fai^{-n}) - \{ 0,k \}\right) \cap F = \varnothing$.

Thus, Let $N=\max \left\{ N_1, N_2 \right\} $
 \begin{align*}
        B_{n,k} \cap F &= \left( (orb(k,\fai^n) - \{ k \}) \cup (orb(k,\fai^{-n}) - \{ 0,k \}) \right) \cap F \\
        & = \left( (orb(k,\fai^n) - \{ k \}) \cap C \right) \cup \left( (orb(k,\fai^{-n}) - \{ 0,k \}) \cap F \right) \\
        & = \varnothing.
\end{align*}
which completes the proof.
\end{proof}

The above lemma shows, for any finite subset of positive integers and a fixed $k$, when $n$ is large enough, the $\fai^n$ orbit of $k$ can always ``escape" from this finite subset.

\subsection{on \texorpdfstring{$\ell^p$}{l\^{}p} space}\leavevmode

In this section, we want to characterize the chaoticity of the unbounded weighted pseudo-shift operator $wC_{\varphi}$ on the space $\ell^p$. and we adjust the notation for $x$, let $x(k)$  denote the number at the $k$-th position of $x$.
\begin{thm}\label{zhouqidianbiaoshi}
        Suppose $\fai$ is injective, and $\fai $ has no fixed point for any $n \in \mathbb{N}$. Then a point $x_N \in D\left[ \left( wC_{\varphi} \right) ^N \right] $ is a $N$-periodic points of $wC_\fai$ if and only if $x_N$ can be written as
        \begin{align*}
            x_N = \sum_{k\in G_N}{a_k x_{N,k}},
        \end{align*}
        where $a_k\in \mathbb{K}$ and
        \begin{align*}
            x_{N,k} = e_k + \sum_{n= 1}^{+ \infty}{\left[\frac{1}{w_k w_{\fai (k)} \cdots w_{\fai^{nN-1}(k)} }\right] e_{\fai^{nN}(k)}}
                          + \sum_{n= 1}^{+ \infty}{\left[w_{\fai^{-1}(k)} w_{\fai^{-2}(k)} \cdots w_{\fai^{-nN}(k)}\right] e_{\fai^{-nN}(k)}}.
        \end{align*} 
\end{thm}
\begin{proof}
Let $x_N =\sum_{k=1}^{\infty}{a_k e_k}$ be an $N$-periodic point of $wC_\fai$, then
\begin{align}
        (wC_\fai)^N (x_N) = (wC_\fai)^N (\sum_{k=1}^{\infty}{a_k e_k}) 
        = \sum_{k=1}^{\infty}{w_k w_{\fai(k)} \cdots w_{\fai^{N-1}(k)} a_{\fai^N (k)} e_k}
        = x_N
        = \sum_{k=1}^{\infty}{a_k e_k}.
\end{align}
Therefore, for any $ k \in \N$,
\begin{align}
        w_k w_{\fai(k)} \cdots w_{\fai^{N-1}(k)} a_{\fai^N (k)}  = a_k,
        \label{zhouqidiedaifan1}
    \end{align}
that is,
\begin{align}
        a_{\fai^N(k)} = \frac{1}{w_k w_{\fai(k)} \cdots w_{\fai^{N-1}(k)}} a_k
        \label{zhouqidiedai1}
\end{align}
Replacing $l = \fai^N(k)$ by  (\ref{zhouqidiedai1}) , we get
 \begin{align*}
        a_{\fai^{2N}(k)} = a_{\fai^N(l)} &= \frac{1}{w_l w_{\fai(l)} \cdots w_{\fai^{N-1}(l)}} a_l\\
        &= \frac{1}{w_{\fai^{N}(k)} w_{\fai^{N+1}(k)} \cdots w_{\fai^{2N-1}(k)}} \frac{1}{w_k w_{\fai(k)} \cdots w_{\fai^{N-1}(k)}} a_k\\
        &= \frac{1}{w_k w_{\fai(k)} \cdots w_{\fai^{2N-1}(k)}} a_k.
    \end{align*}
By induction, $\forall n \in \N$,
\begin{align}
        a_{\fai^{nN}(k)} = \frac{1}{w_k w_{\fai(k)} \cdots w_{\fai^{nN-1}(k)} } a_k.
        \label{zhouqidiedai2}
\end{align}
We note that $l = \fai^N(k)$, then $k = \fai^{-N}(l)$, Substitute it into (\ref{zhouqidiedaifan1}),
\begin{align*}
        a_{\fai^{-N}(l)}
        = [ w_{\fai^{-N}(l)} w_{\fai^{-N+1}(l)} \cdots w_{\fai^{-1}(l)} ] a_l 
        = [ w_{\fai^{-1}(l)} w_{\fai^{-2}(l)} \cdots w_{\fai^{-N}(l)} ] a_l .
\end{align*}
By induction, $\forall n \in \N$
\begin{align}
        a_{\fai^{-nN}(k)} 
        = [ w_{\fai^{-1}(k)} w_{\fai^{-2}(k)} \cdots w_{\fai^{-nN}(k)} ] a_k.
        \label{zhouqidiedaifan2}
    \end{align}
Hence for any $ k,n\in \N$, by (\ref{zhouqidiedai2}) , (\ref{zhouqidiedaifan2}) and Lemma \ref{guidaobianli}, 
    \begin{align*}
        x_N & = \sum_{k=1}^{\infty}{a_k e_k} 
              = \sum_{k \in G_N}{\left[ a_k e_k + \sum_{n=1}^{\infty}{a_{\fai^{nN}(k)} e_{\fai^{nN}(k)}} + \sum_{n=1}^{\infty}{a_{\fai^{-nN}(k)} e_{\fai^{-nN}(k)}} \right] }\\
            & = \sum_{k \in G_N}{\left[ a_k e_k + \sum_{n=1}^{\infty}{\frac{1}{w_k w_{\fai(k)} \cdots w_{\fai^{nN-1}(k)} } a_k e_{\fai^{nN}(k)}} 
                                                + \sum_{n=1}^{\infty}{[ w_{\fai^{-1}(k)} w_{\fai^{-2}(k)} \cdots w_{\fai^{-nN}(k)} ] a_k e_{\fai^{-nN}(k)}} \right] } \\
            & = \sum_{k \in G_N}{a_k \left[e_k + \sum_{n= 1}^{+ \infty}{\left[\frac{1}{w_k w_{\fai (k)} \cdots w_{\fai^{nN-1}(k)} }\right] e_{\fai^{nN}(k)}}
                          + \sum_{n= 1}^{+ \infty}{\left[w_{\fai^{-1}(k)} w_{\fai^{-2}(k)} \cdots w_{\fai^{-nN}(k)}\right] e_{\fai^{-nN}(k)}} \right] } \\
            & = \sum_{k \in G_N}{a_k x_{N,k}}.
    \end{align*}

As to the sufficient part of the proof, for any $k\in G_N$
     \begin{align*}
        wC_\fai ^{N}(x_{N,k}) 
        &=  wC_\fai ^{N}(e_k)  \\
        & \  \  + \sum_{n= 1}^{\infty}{\left[\frac{1}{w_k w_{\fai (k)} \cdots w_{\fai^{nN-1}(k)} }\right] wC_\fai ^{N}(e_{\fai^{nN}(k)} ) } \\
        & \  \  + \sum_{n= 1}^{\infty}{\left[w_{\fai^{-1}(k)} w_{\fai^{-2}(k)} \cdots w_{\fai^{-nN}(k)}\right] wC_\fai ^{N}(e_{\fai^{-nN}(k)})}  \\
        &=  w_{\fai^{-N}(k)} w_{\fai^{-(N-1)}(k)} \cdots w_{\fai^{-1}(k)} e_{\fai^{-N}(k)} \\
        & \  \  + \sum_{n= 1}^{\infty}{\left[\frac{1}{w_k w_{\fai (k)} \cdots w_{\fai^{nN-1}(k)} }\right] 
                   \left[ w_{\fai^{nN-N}(k)} w_{\fai^{nN-(N-1)}(k)} \cdots w_{\fai^{nN-1}(k)} \right] e_{\fai^{nN-N}(k)}}  \\
        & \  \  + \sum_{n= 1}^{\infty}{\left[w_{\fai^{-1}(k)} w_{\fai^{-2}(k)} \cdots w_{\fai^{-nN}(k)}\right] 
                   \left[ w_{\fai^{-nN-N}(k)} w_{\fai^{-nN-(N-1)}(k)} \cdots w_{\fai^{-nN-1}(k)} \right] e_{\fai^{-nN-N}(k)}}  \\
        &=  w_{\fai^{-1}(k)} w_{\fai^{-2}(k)} \cdots w_{\fai^{-N}(k)} e_{\fai^{-N}(k)} + e_k\\
        & \  \  + \sum_{n= 2}^{\infty}{\left[\frac{1}{w_k w_{\fai (k)} \cdots w_{\fai^{(n-1)N-1}(k)} }\right] e_{\fai^{(n-1)N}(k)}} \\
        & \  \  + \sum_{n= 1}^{\infty}{\left[w_{\fai^{-1}(k)} w_{\fai^{-2}(k)} \cdots w_{\fai^{-(n+1)N}(k)}\right] e_{\fai^{-(n+1)N}(k)}}  \\
        &=  e_k + \sum_{n= 1}^{\infty}{\left[\frac{1}{w_k w_{\fai (k)} \cdots w_{\fai^{nN-1}(k)} }\right] e_{\fai^{nN}(k)}} 
                + \sum_{n= 1}^{\infty}{\left[w_{\fai^{-1}(k)} w_{\fai^{-2}(k)} \cdots w_{\fai^{-nN}(k)}\right] e_{\fai^{-nN}(k)}}  \\
        &= x_{N,k}.
    \end{align*}
It implies that each $x_{N,k}$ is an N-periodic point of $wC_\fai$, thus $x_N = \sum_{k\in G_N}{a_k x_{N,k}}$ is also an N-periodic point of $wC_\fai$.
\end{proof}

\begin{rem}
It is easy to see that for any $k\in G_N$ and a fixed $N$, $x_{N,k}\left( j \right) \ne 0$ only when $j\in A_{N,k}$. And Lemma \ref{guidaobujiao} proves that the different orbits $A_{N,k}$ of different $k$ are disjoint, so that any two distinct $x_{N,k}$ have different non-zero locations, that is, for any $k_1,k_2$,$n \in \N$, $x_{N,k_1}(n) x_{N,k_2}(n) =0$. Thus they are always linearly independent.
\end{rem}


Theorem \ref{zhouqidianbiaoshi} actually illustrates the structure of the set of periodic points of $wC_\fai$: every periodic point of period $N$ can be expressed as linear combinations of countable $x_{N,k}$($k \in G_N$).
This can be interpreted as the existence of a set of ``basis" (i.e. ${ x_{N,k} | k \in G_N}$) for the set of periodic points with a fixed period $N$. To verify the denseness of the set of all periodic points is to compare 
the union of bases for all sets of periodic points (i.e. $\bigcup_{N \in \N}{ \{ x_{N,k} | k \in G_N\} }$) and the bases of $\ell^p$ (e.g. $\{ e_n | n\in \N \}$) in terms of how closely they are.

Next, we will show the necessary and sufficient condition for the chaotic weighted pseudo-shift operator~$wC_{\varphi}$ on the space $\ell^p$.

\begin{thm}\label{lphundun}
         Let $wC_{\varphi} $~be an unbounded weighted pseudo-shift operator on $\ell^p$, then $wC_{\varphi} $ is chaotic if and only if the following assertions hold:  \\
        \indent $(1)$ $\varphi $ is injective.\\
        \indent $(2)$  $\varphi ^n$ has no fixed point for any $n \in \mathbb{N} $.\\
        \indent $(3)$  $\forall k \in \N$, the following inequalities hold
        \begin{equation*}
        \sum_{n=1}^{+\infty}{\frac{1}{\left| w_kw_{\varphi \left( k \right)}...w_{\varphi ^{n-1}\left( k \right)} \right|^p}}<+\infty ,
        \end{equation*}
        and
        \begin{equation*}
        \sum_{n=1}^{+\infty}{\left| w_{\varphi ^{-1}\left( k \right)}w_{\varphi ^{-2}\left( k \right)}...w_{\varphi ^{-n}\left( k \right)} \right|}^p<+\infty.
        \end{equation*} 
\end{thm}
\begin{proof}
For the necessary part of the proof, assume that $wC_\fai$ is chaotic, it is naturally also hypercyclic, then the condition $(1)$ and $(2)$ has been proved in the Theorem \ref{wujiew}, so we need only to show the condition $(3)$. Suppose $x_N \in D\left[ \left( wC_{\varphi} \right) ^N \right]$ is a periodic point of $wC_\fai$ with period $N$,  by Theorem \ref{zhouqidianbiaoshi},
\begin{align*}
        x_N &= \sum_{k\in G_N}{a_k x_{N,k}} \\
            &= \sum_{k\in G_N}{a_k \left[e_k + \sum_{n= 1}^{+ \infty}{\left[\frac{1}{w_k w_{\fai (k)} \cdots w_{\fai^{nN-1}(k)} }\right] e_{\fai^{nN}(k)}}
                   + \sum_{n= 1}^{+ \infty}{\left[w_{\fai^{-1}(k)} w_{\fai^{-2}(k)} \cdots w_{\fai^{-nN}(k)}\right] e_{\fai^{-nN}(k)}}\right]}.
\end{align*}
Then for any $i=0,1,2,\cdots,N-1$,
 \begin{align*}
        wC_\fai ^{i}(x_N) 
        &= \sum_{k\in G_N}{a_k \left[wC_\fai ^{i}\left(x_{N,k}\right)\right]} \\
        &= \sum_{k\in G_N}{a_k \left[\sum_{n= 1}^{+ \infty}{\left[\frac{1}{w_k w_{\fai (k)} \cdots w_{\fai^{nN-(i+1)}(k)} }\right] e_{\fai^{nN-i}(k)}}
        + \sum_{n= 0}^{+ \infty}{\left[w_{\fai^{-1}(k)} w_{\fai^{-2}(k)} \cdots w_{\fai^{-(nN+i)}(k)}\right] e_{\fai^{-(nN+i)}(k)}}\right]}.
\end{align*}
Since $x_{N}$ is an $N$-periodic point of $wC_\fai$, it follows that $wC_\fai ^{i}(x_N) \in l^p.$ Therefore $\forall k \in G_N$,
 \begin{align}
        \sum_{n=1}^{+\infty}{\frac{1}{\left| w_kw_{\varphi \left( k \right)}...w_{\varphi ^{nN-(i+1)}\left( k \right)} \right|^p}}<+\infty,
        \label{chongfen1shi}
    \end{align}
     \begin{align}
        \sum_{n=1}^{+\infty}{\left| w_{\varphi ^{-1}\left( k \right)}w_{\varphi ^{-2}\left( k \right)}...w_{\varphi ^{-(nN+i)}\left( k \right)} \right|}^p<+\infty.
        \label{chongfen2shi}
    \end{align}
Since (\ref{chongfen1shi}) and (\ref{chongfen2shi}) hold for all $i=0,1,2,\cdots,N-1$, then $\forall k \in G_N$,
\begin{align*}
        \sum_{n=1}^{+\infty}{\frac{1}{\left| w_kw_{\varphi \left( k \right)}...w_{\varphi ^{n-1}\left( k \right)} \right|^p}}<+\infty ,\\
        \sum_{n=1}^{+\infty}{\left| w_{\varphi ^{-1}\left( k \right)}w_{\varphi ^{-2}\left( k \right)}...w_{\varphi ^{-n}\left( k \right)} \right|}^p<+\infty.
\end{align*}
For all $l \in \N$, suppose $k=g(l)$, there exists $m\in \mathbb{Z}$ such that $l=\varphi ^m\left( k \right)$. By Lemma \ref{Glian}, $k \in G \subseteq G_N$, then
\begin{align*}
      \sum_{n=1}^{+\infty}{\frac{1}{\left| w_lw_{\varphi \left( l \right)}...w_{\varphi ^{n-1}\left( l \right)} \right|^p}} &= \left| w_kw_{\varphi \left( k \right)}...w_{\varphi ^{m-1}\left( k \right)} \right| \sum_{n=1}^{+\infty}{\frac{1}{\left| w_kw_{\varphi \left( k \right)}...w_{\varphi ^{n+m-1}\left( k \right)} \right|^p}}<+\infty ,\\
        \sum_{n=1}^{+\infty}{\left| w_{\varphi ^{-1}\left( l \right)}w_{\varphi ^{-2}\left( l \right)}...w_{\varphi ^{-n}\left( l \right)} \right|}^p &= \frac{1}{\left| w_kw_{\varphi \left( k \right)}...w_{\varphi ^{m-1}\left( k \right)} \right|} \sum_{n=1}^{+\infty}{\left| w_{\varphi ^{-1}\left( k \right)}w_{\varphi ^{-2}\left( k \right)}...w_{\varphi ^{-n+m}\left( k \right)} \right|}^p<+\infty.
    \end{align*}

For the proof of sufficiency, assume that the condition $(1)$,$(2)$ and $(3)$ hold, then $wC_\fai$ is hypercyclic by Theorem \ref{wujiew}. Now it suffices to show that $wC_\fai$  has a dense set of periodic points.

For any vector $y \in c_{00}$, there exists $m \in \mathbb{N}$ such that $y\left( k \right)=0$ for any $ k>m$, that is,
\begin{align*}
        y = (y_1,y_2,\cdots,y_m,0,\cdots) = \sum_{k=1}^{m}{y_k e_k}.
\end{align*}
By lemma \ref{taoyi}, for each $k=1,2,\cdots,m$, there exists $N_k$ such that $ B_{n,k} \cap \{1,2,\cdots,m \} = \varnothing $ for any $n > N_k$. Take $N_a=\max_{k=1,2,\cdots,m}\{N_k\}$, then for each $k=1,2,\cdots,m$, $ B_{n,k} \cap \{1,2,\cdots,m \} = \varnothing $ for any $n > N_a$. Hence, 
\begin{align*}
        A_{n,k} \cap \{1,2,\cdots,m \} = k,
\end{align*}
that is, the intersection of $\{1,2,\cdots,m \}$ with the $\fai^{n}$-orbit $A_{n,k}$ is only $\left\{ k \right\} $.

By the condition (3), $\forall \varepsilon >0$, there exists $N_b>0$ such that 
\begin{align}
        \sum_{n>N_a}^{}{\frac{1}{\left| w_kw_{\varphi \left( k \right)}...w_{\varphi ^{n-1}\left( k \right)} \right|^p}}<\frac{\varepsilon}{2m} 
        \label{shoulian1},\\
        \sum_{n>N_b}^{}{\left| w_{\varphi ^{-1}\left( k \right)}w_{\varphi ^{-2}\left( k \right)}...w_{\varphi ^{-n}\left( k \right)} \right|}^p<\frac{\varepsilon}{2m}.
        \label{shoulian2}
\end{align}
Let $N=\max\left\{N_a,N_b\right\}$, consider
\begin{align*}
        x_{N,k} = e_k + \sum_{n= 1}^{+ \infty}{\left[\frac{1}{w_k w_{\fai (k)} \cdots w_{\fai^{nN-1}(k)} }\right] e_{\fai^{nN}(k)}}
                      + \sum_{n= 1}^{+ \infty}{\left[w_{\fai^{-1}(k)} w_{\fai^{-2}(k)} \cdots w_{\fai^{-nN}(k)}\right] e_{\fai^{-nN}(k)}},
    \end{align*}
the above discussion ensures that only the $k$-th term among the first $m$ terms of $x_{N,k}$ is not 0.

First, we show that $x_{N,k}$ is well-defined and is a periodic point. We have proved that $\left( wC_{\varphi} \right)^N\left(x_{N,k}\right) =x_{N,k}$ in Theorem \ref{zhouqidianbiaoshi}. By condition (3),
\begin{align*}
         \lVert x_{N,k} \rVert &=\left( 1+\sum_{n=1}^{+\infty}{\frac{1}{\left| w_kw_{\varphi \left( k \right)}...w_{\varphi ^{nN-1}\left( k \right)} \right|^p}}+\sum_{n=1}^{+\infty}{\left| w_{\varphi ^{-1}\left( k \right)}w_{\varphi ^{-2}\left( k \right)}...w_{\varphi ^{-nN}\left( k \right)} \right|}^p \right) ^{\frac{1}{p}} \\
         &\le \left( 1+\sum_{n=1}^{+\infty}{\frac{1}{\left| w_kw_{\varphi \left( k \right)}...w_{\varphi ^{n-1}\left( k \right)} \right|^p}}+\sum_{n=1}^{+\infty}{\left| w_{\varphi ^{-1}\left( k \right)}w_{\varphi ^{-2}\left( k \right)}...w_{\varphi ^{-n}\left( k \right)} \right|}^p \right) ^{\frac{1}{p}}\\
         &< +\infty,
\end{align*}
hence $\left( wC_{\varphi} \right) ^N\left( x_{N,k} \right) =x_{N,k}$ and $x_{N,k} \in  D\left( \left[ wC_{\fai}\right]^{N}  \right) $. 

Next, We prove the denseness of the set of periodic points. Let
\begin{align*}
        x = \sum_{k=1}^{m}{y_k x_{N,k}},
    \end{align*}
as is easily seen, $x$ is also the $N$-periodic point of $wC_\fai$. Suppose $y_0=\max_{k=1,2,\cdots,m}{|y_k|}$, by $\left( \ref{shoulian1}\right)$,$\left( \ref{shoulian2} \right)$ and the choice of $N$,
 \begin{align*}
         & \text{\quad \, } \lVert x - y \rVert \\
         & = \left\lVert \sum_{k=1}^{m}{y_k \left[ \sum_{n= 1}^{+ \infty}{\left[\frac{1}{w_k w_{\fai (k)} \cdots w_{\fai^{nN-1}(k)} }\right] e_{\fai^{nN}(k)}}
          + \sum_{n= 1}^{+ \infty}{\left[w_{\fai^{-1}(k)} w_{\fai^{-2}(k)} \cdots w_{\fai^{-nN}(k)}\right] e_{\fai^{-nN}(k)}}) \right] } \right\rVert \\
          & \le \left\lVert \sum_{k=1}^{m}{y_k  \sum_{n= 1}^{+ \infty}{\left[\frac{1}{w_k w_{\fai (k)} \cdots w_{\fai^{nN-1}(k)} }\right] e_{\fai^{nN}(k)}}
          }\right\rVert+ \left\lVert{\sum_{k=1}^{m}{y_k\sum_{n= 1}^{+ \infty}{\left[w_{\fai^{-1}(k)} w_{\fai^{-2}(k)} \cdots w_{\fai^{-nN}(k)}\right] e_{\fai^{-nN}(k)}})  } }\right\rVert \\   
          & \le y_0\sum_{k=1}^{m}{\left\lVert   \sum_{n= 1}^{+ \infty}{\left[\frac{1}{w_k w_{\fai (k)} \cdots w_{\fai^{nN-1}(k)} }\right] e_{\fai^{nN}(k)}}
          \right\rVert}+ y_0\sum_{k=1}^{m}{\left\lVert{\sum_{n= 1}^{+ \infty}{\left[w_{\fai^{-1}(k)} w_{\fai^{-2}(k)} \cdots w_{\fai^{-nN}(k)}\right] e_{\fai^{-nN}(k)}})  } \right\rVert} \\  
          & <\frac{y_0}{2}\varepsilon +\frac{y_0}{2}\varepsilon=y_0 \varepsilon.
    \end{align*}

Hence, in view of the denseness of $c_{00}$ in $\ell^p$, we infer that the set of periodic points of $wC_\fai$  is dense in $\ell^p$ as well, i.e, the operator $wC_\fai$ is chaotic.
\end{proof}

\begin{ex}
Let the mapping $\fai$ be defined by $\varphi \left( n \right) =\begin{cases}
	4k+2,&		n=4k\\
	4k-3,&		n=4k+1 ~~~~~~ \left(k\ge 1\right) \\
        n+1,&     \text{else}
        \end{cases}$ 
    and let $w=\left( w_n \right) _{n\in \mathbb{N}}$ with $w_n =\begin{cases}
	w^{-n},&		n=4k+1 ~~~~~~ \left(k\ge 1\right) \\
        w^n,&     \text{else}
        \end{cases}$ 
    where $\left| w \right|>1$. Then one easily verifies that $\fai$ is injective, $\varphi ^n$ has no fixed point for any $n \in \mathbb{N} $ and 
\begin{subequations}
\begin{align*}
        \sum_{n=1}^{+\infty}{\frac{1}{\left| w_kw_{\varphi \left( k \right)}...w_{\varphi ^{n-1}\left( k \right)} \right|^p}}<+\infty ,\\
        \sum_{n=1}^{+\infty}{\left| w_{\varphi ^{-1}\left( k \right)}w_{\varphi ^{-2}\left( k \right)}...w_{\varphi ^{-n}\left( k \right)} \right|}^p<+\infty.
\end{align*}
\end{subequations}
   for each $k\in \mathbb{N}$.  Hence, by the Theorem \ref{lphundun}, the unbounded operator $wC_\fai$ is chaotic.   
\end{ex}

\subsection{on \texorpdfstring{$c_0$} a  space}\leavevmode

In this section, we extend the characterization of the unbounded chaotic weighted pseudo-shift operator to the space $c_0$.

\begin{thm}\label{hundunchongyao}
         Let $wC_{\varphi} $~be an unbounded weighted pseudo-shift operator on $c_0$, then $wC_{\varphi} $ is chaotic if and only if the following assertions hold:  \\
        \indent $(1)$ $\fai$ is injective. \\
        \indent $(2)$ $\fai ^n$ has no fixed point for any $n \in \mathbb{N} $.\\
        \indent $(3)$ For each $k \in \N$, 
        \begin{equation*}
        \lim_{n\rightarrow \infty}\left| w\left( k \right) w\left( \fai \left( k \right) \right) \cdots w\left( \fai ^{n-1}\left( k \right) \right) \right|=\infty,
        \end{equation*}
        and
        \begin{equation*}
        \lim_{n\rightarrow \infty}\left| w(\fai^{-1}(k)) w(\fai^{-2}(k)) \cdots w(\fai^{-n}(k)) \right|=0.
        \end{equation*}
    \end{thm}
\begin{proof}
For the necessary part of the proof, assume that $wC_\fai$ is chaotic, it is naturally also hypercyclic, then the condition $(1)$ and $(2)$ has been proved in the Corollary \ref{c0chaoxunhuan}, so we need only to show the condition $(3)$. Suppose $x_N \in D\left[ \left( wC_{\varphi} \right) ^N \right]$ is a periodic point of $wC_\fai$ with period $N$,  by Theorem  \ref{zhouqidianbiaoshi},
\begin{align*}
        x_N &= \sum_{k\in G_N}{a_k x_{N,k}} \\
            &= \sum_{k\in G_N}{a_k \left[e_k + \sum_{n= 1}^{+ \infty}{\left[\frac{1}{w_k w_{\fai (k)} \cdots w_{\fai^{nN-1}(k)} }\right] e_{\fai^{nN}(k)}}
                   + \sum_{n= 1}^{+ \infty}{\left[w_{\fai^{-1}(k)} w_{\fai^{-2}(k)} \cdots w_{\fai^{-nN}(k)}\right] e_{\fai^{-nN}(k)}}\right]}.
\end{align*}
Then for any $i=0,1,2,\cdots,N-1$
 \begin{align*}
        wC_\fai ^{i}(x_N) 
        &= \sum_{k\in G_N}{a_k \left[wC_\fai ^{i}\left(x_{N,k}\right)\right]} \\
        &= \sum_{k\in G_N}{a_k \left[\sum_{n= 1}^{+ \infty}{\left[\frac{1}{w_k w_{\fai (k)} \cdots w_{\fai^{nN-(i+1)}(k)} }\right] e_{\fai^{nN-i}(k)}}
        + \sum_{n= 0}^{+ \infty}{\left[w_{\fai^{-1}(k)} w_{\fai^{-2}(k)} \cdots w_{\fai^{-(nN+i)}(k)}\right] e_{\fai^{-(nN+i)}(k)}}\right]}.
\end{align*}
Since $x_{N}$ is an N-periodic point of $wC_\fai$ , it follows that $wC_\fai ^{i}(x_N) \in c_0$. And we know that for different $wC_\fai ^{i}(x_{N,k})$, the non-zero positions are different. By Lemma \ref{faiwujie}, $\fai^{nN-i}(k) \you \infty \ (n \you \infty)$, then $\forall k \ \in G_N$
\begin{align}
        \frac{1}{\left |w_k w_{\fai (k)} \cdots w_{\fai^{nN-(i+1)}(k)} \right|}  \you 0, \  \ ( n \you \infty)
        \label{chongfen1shi2}
\end{align}
If $\forall j \in \N$, $\fai^{-j}(k) \ne 0$, we infer by Lemma \ref{faiwujie2} that $\fai^{-(nN+i)}(k) \you \infty \ (n \you \infty)$, and since  $wC_\fai ^{i}(x_{N,k}) \in c_0$, then $\forall k \ \in G_N$
\begin{align}
         w_{\fai^{-1}(k)} w_{\fai^{-2}(k)} \cdots w_{\fai^{-(nN+i)}(k)}  \you 0, \  \ ( n \you \infty)
         \label{chongfen2shi2}
    \end{align}
If $\exists j_1 \in \N$, $\fai^{-j_1}(k) = 0$ (which implies that $j >j_1$, $\fai^{-j}(k) = 0$ ), by the nonation in (\ref{notation}), $w_{\fai^{-j}(k)} = w_0 =0$. Thus, when $n$ is large enough , it follows that $w_{\fai^{-1}(k)} w_{\fai^{-2}(k)} \cdots w_{\fai^{-(nN+i)}(k)} = 0$ so that (\ref{chongfen2shi2}) holds.

since (\ref{chongfen1shi2}) and (\ref{chongfen2shi2}) hold for all $i=0,1,2,\cdots,N-1$, then $\forall  k \in G_N$
\begin{align*}
        \lim_{n\rightarrow \infty}\left| w\left( k \right) w\left( \fai \left( k \right) \right) \cdots w\left( \fai ^{n-1}\left( k \right) \right) \right|=\infty,\\
        \lim_{n\rightarrow \infty}\left| w(\fai^{-1}(k)) w(\fai^{-2}(k)) \cdots w(\fai^{-n}(k)) \right|=0.
    \end{align*}
For all $l \in \N$, suppose $k=g(l)$, then there exists $m\in\mathbb{Z}$ such  that $l=\varphi ^m\left( k \right)$. By Lemma \ref{Glian}, $k \in G \subseteq G_N$, then
\begin{align*}
        \lim_{n\rightarrow \infty}\left| w\left( l \right) w\left( \fai \left( l \right) \right) \cdots w\left( \fai ^{n-1}\left( l \right) \right) \right|
      &= \frac{1}{\left| w_kw_{\varphi \left( k \right)}...w_{\varphi ^{m-1}\left( k \right)} \right|} \lim_{n\rightarrow \infty}\left| w\left( k \right) w\left( \fai \left( k \right) \right) \cdots w\left( \fai ^{n+m-1}\left( k \right) \right) \right| = \infty,\\
        \lim_{n\rightarrow \infty}\left| w(\fai^{-1}(l)) w(\fai^{-2}(l)) \cdots w(\fai^{-n}(l)) \right|
      &=\left| w_kw_{\varphi \left( k \right)}...w_{\varphi ^{m-1}\left( k \right)} \right| \lim_{n\rightarrow \infty}\left| w(\fai^{-1}(k)) w(\fai^{-2}(k)) \cdots w(\fai^{-n+m}(k)) \right| =0.
\end{align*}

For the proof of sufficiency, assume that the condition $(1)$,$(2)$ and $(3)$ hold,  then $wC_\fai$ is hypercyclic by Corollary \ref{c0chaoxunhuan}. Now it suffices to show that $wC_\fai$  has a dense set of periodic points.
For any vector $y \in c_{00}$, there exists $m \in \mathbb{N}$ such that $y\left( k \right)=0$ for any $ k>m$, that is,
\begin{align*}
        y = (y_1,y_2,\cdots,y_m,0,\cdots) = \sum_{k=1}^{m}{y_k e_k}.
\end{align*}
By lemma \ref{taoyi}, for each $k=1,2,\cdots,m$, there exists $N_k$ such that $ B_{n,k} \cap \{1,2,\cdots,m \} = \varnothing $ for any $n > N_k$. Take $N_0=\max_{k=1,2,\cdots,m}\{N_k\}$, then $ B_{n,k} \cap \{1,2,\cdots,m \} = \varnothing $ for each $k=1,2,\cdots,m$ and any $n > N_0$. Hence, 
\begin{align*}
        A_{n,k} \cap \{1,2,\cdots,m \} = k,
\end{align*}
that is, the intersection of $\{1,2,\cdots,m \}$ with the $\fai^{n}$ orbit $A_{n,k}$ is only $\left\{ k \right\} $.
For any $N>N_0$, consider
\begin{align*}
        x_{N,k} = e_k + \sum_{n= 1}^{+ \infty}{\left[\frac{1}{w_k w_{\fai (k)} \cdots w_{\fai^{nN-1}(k)} }\right] e_{\fai^{nN}(k)}}
                      + \sum_{n= 1}^{+ \infty}{\left[w_{\fai^{-1}(k)} w_{\fai^{-2}(k)} \cdots w_{\fai^{-nN}(k)}\right] e_{\fai^{-nN}(k)}}.
    \end{align*}
It is clear that $x_{N,k}$ is an $N$-periodic point of $wC_\fai$ and the above discussion ensures that only the $k$-th term among the first $m$ terms of $x_{N,k}$ is not 0.

Let us first show  $x_{N,k}$ is well-defined. We note that for any $l>m$,
$$
x_{N,k}\left( l \right) =\begin{cases}
	\left| \frac{1}{w_kw_{\varphi \left( k \right)}\cdots w_{\varphi ^{nN-1}\left( k \right)}} \right|&		\text{,}\exists n\in \mathbb{N}\,\,s.t.\,\,l=\varphi ^{nN}\left( k \right) \text{,}\\
	\left| w_{\varphi ^{-1}\left( k \right)}w_{\varphi ^{-2}\left( k \right)}\cdots w_{\varphi ^{-nN}\left( k \right)} \right|&		\text{,}\exists n\in \mathbb{N}\,\,s.t.\,\,l=\varphi ^{-nN}\left( k \right) \text{,}\\
	0&		\text{,\,\,else.}\\
\end{cases}
$$
It follows from condition $(3)$ that for any $\forall \varepsilon >0$, there exsits $N_1>0$ such that for each $n>N_1$,
\begin{align*}
   \frac{1}{\left| w_kw_{\varphi \left( k \right)}\cdots w_{\varphi ^{nN-1}\left( k \right)} \right|}<\varepsilon,\\ 
   \left| w_{\varphi ^{-1}\left( k \right)}w_{\varphi ^{-2}\left( k \right)}\cdots w_{\varphi ^{-nN}\left( k \right)} \right|<\varepsilon. 
\end{align*}
 If there exists $ n\in \mathbb{N}\,\,s.t.\,\,l=\varphi ^{nN}\left( k \right)$, then it follows from lemma \ref{n1n2} that there exists $l_1\in \mathbb{N}$ such that for any $l>l_1$,  $n>N_1$. If there exists $ n\in \mathbb{N}\,\,s.t.\,\,l=\varphi ^{-nN}\left( k \right)$, then it follows from lemma \ref{n1n2} that there exists $l_2\in \mathbb{N}$ such that for any $l>l_2$,  $n>N_1$. Let $l_0=\max \left\{ l_1,l_2 \right\}$, then for any $l>l_0$, we have $x_{N,k}\left( l \right) <\varepsilon$ which implies  $x_{N,k}\in c_0$.

Next, We prove the denseness of the set of periodic points. Let
\begin{align*}
        x = \sum_{k=1}^{m}{y_k x_{N,k}},
\end{align*}
As is easily seen, $x$ is also the $N$-periodic point of $wC_\fai$. Suppose $y_0=\max_{k=1,2,\cdots,m}{|y_k|}$, then when $n \you 0$,
by condition $(3)$
\begin{align*}
        & \text{\quad \, } \lVert x - y \rVert \\
        & = \left\lVert \sum_{k=1}^{m}{y_k \left[ \sum_{n= 1}^{+ \infty}{\left[\frac{1}{w_k w_{\fai (k)} \cdots w_{\fai^{nN-1}(k)} }\right] e_{\fai^{nN}(k)}}
          + \sum_{n= 1}^{+ \infty}{\left[w_{\fai^{-1}(k)} w_{\fai^{-2}(k)} \cdots w_{\fai^{-nN}(k)}\right] e_{\fai^{-nN}(k)}}) \right] } \right\rVert \\
        & \le y_0 \max \left\{ \sup_{n\in \N}{\left| \frac{1}{w_k w_{\fai (k)} \cdots w_{\fai^{nN-1}(k)} }\right|},\  \sup_{n\in \N}{\left| w_{\fai^{-1}(k)} w_{\fai^{-2}(k)} \cdots w_{\fai^{-nN}(k)}\right|} \right\}
        \you 0.
\end{align*}
Hence,  in view of the denseness of $c_{00}$ in $c_0$, we infer that the set of periodic points of $wC_\fai$ is dense in $c_0$ as well,i.e, the operator $wC_\fai$ is chaotic.
\end{proof}

\begin{ex}
Let the mapping $\fai :\mathbb{N}\you \mathbb{N}$ defined in the example \ref{sushulizi} (2)
    and let  $w_n =\begin{cases}
	w^{-n},&		n=P_i\left( 3k+1 \right)   \\
        w^n,&     \text{else}
        \end{cases}$ 
        for all $k\ge1$ and $i\ge 0$ where $\left| w \right|>1$. Then one easily verifies that $\fai$ is injective, $\varphi ^n$ has no fixed point for any $n \in \mathbb{N} $ and 
\begin{subequations}
\begin{align*}
        \lim_{n\rightarrow \infty}\left| w\left( k \right) w\left( \fai \left( k \right) \right) \cdots w\left( \fai ^{n-1}\left( k \right) \right) \right|=\infty,\\
        \lim_{n\rightarrow \infty}\left| w(\fai^{-1}(k)) w(\fai^{-2}(k)) \cdots w(\fai^{-n}(k)) \right|=0.
\end{align*}
\end{subequations}
   for each $k\in \mathbb{N}$.  Hence, by the Theorem \ref{hundunchongyao}, the unbounded operator $wC_\fai$ is chaotic.
\end{ex}

\bibliographystyle{unsrt}
\bibliography{cite}

\end{document}